\documentclass[12pt]{amsart}
\usepackage[a4paper, total={6in, 8in}]{geometry}
\newtheorem{theorem}{Theorem}[section]
\newtheorem{lemma}[theorem]{Lemma}

\theoremstyle{definition}
\newtheorem{definition}[theorem]{Definition}
\newtheorem{example}[theorem]{Example}

\theoremstyle{remark}
\newtheorem{remark}[theorem]{Remark}

\numberwithin{equation}{section}

\begin{document}

\title[Application of measures of noncompactness to BVP]{Boundary value problem for an infinite system of second order differential equations in $\ell_p$ spaces}


\author{Ishfaq Ahmad Malik}
\address{Department of Mathematics, National Institute of Technology, Hazratbal, Srinagar-190006, Jammu and Kashmir, India.}
\curraddr{}
\email{ishfaq$\_$2phd15@nitsri.net}

\author{Tanweer Jalal}
\address{Department of Mathematics, National Institute of Technology, Hazratbal, Srinagar-190006, Jammu and Kashmir, India.}
\curraddr{}
\email{tjalal@nitsri.net}
\thanks{}



\date{}

\dedicatory{}

\begin{abstract}
In this paper the concept of measure of noncompactness is applied to prove the existence of solution for a boundary value problem for an infinite system of second order differential equations in $\ell_{p}$ space. We change the boundary value problem into an equivalent system of infinite integral equations and obtain the result for the system of integral equations, using Darbo type fixed point theorem. The result is applied to an example to illustrate the concept.
\end{abstract}

\maketitle

\section{Introduction and Preliminaries}
In 1930 Kuratowski \cite{kuratowski} introduced the concept of measure of noncompactness which was further extendend to general Banach space by Bana$\acute{s}$ and Goebel \cite{Josaf}. In 1955 Darbo \cite{darbo55} proved a fixed point theorem for condensing operators using the concept of measure of noncompactness, which generalized the classical Schauder fixed point theorem and Banach contraction principle. The method of fixed point arguments has been widely used to study the existence of solutions of functional equations, like Banach contraction principle in \cite{ agarwal10s,olaru10} and Schauder's fixed point theorem in \cite{klamka00schauder, liu2001applications}. But if compactness and Lipschitz condition are not satisfied these results can not be used. Measure of noncompactness comes handy in such situations.\\

The Hausdorff measure of noncompactness is used frequently in finding the existence of solutions for various functional equations  and is defined as: 
\begin{definition} \cite{Josaf} Let $(\Omega,d)$ be a metric space and $A$ be a bounded subset of $\Omega$. Then the \emph{Hausdorff measure of noncompactness} (the ball-measure of noncompactness) of the set $A$, denoted by $\chi(A)$ is defined to be the infimum of the set of all real $\epsilon>0$ such that $A$ can be covered by a finite number of balls of radii $<\epsilon$, that is 
\begin{align*}
\chi(A)=\inf \left\{\epsilon>0:A\subset \bigcup_{i=1}^{n}B(x_{i},r_{i}),x_{i}\in\Omega, r_{i}<\epsilon~(i=1,\ldots, n)~,n\in\mathbb{N} \right\}
\end{align*}     
where $B(x_{i},r_{i})$ denotes ball of radius $r_i$ centered at $x_i$.
\end{definition}

Let $\left(X,\|\cdot \|\right)$ be a Banach space, for any $E\subset X$, $\bar{E}$ denotes closure of $E$ and $conv(E)$ denotes the closed convex hull of $E$. We denote the family of non-empty bounded subsets of $X$ by $\mathfrak{M}_{X}$ and family of non-empty and relatively compact subsets of $X$ by $\mathfrak{N}_{X}$. Let $\mathbb{N}$ denote the set of natural numbers and $\mathbb{R}$ the set of real numbers for $\mathbb{R}_{+}=[0,\infty)$ the axiomatic definition of measure of noncompactness is defined below

\begin{definition}\cite{M1-b} A mapping $\mu:\mathfrak{M}_{X}\rightarrow\mathbb{R}_{+}$ is said to be the measure of noncompactness in $E$ if the following conditions hold:
\begin{enumerate}
\item[(i)] The family $\text{Ker }\mu=\{E\in\mathfrak{M}_{X}:\mu(E)=0\}$ is non-empty and $\text{Ker }\mu \subset \mathfrak{N}_{X}$;
\item[(ii)] $E_1\subset E_2 \Rightarrow \mu(E_1)\leq \mu(E_2)$;
\item[(iii)] $\mu(\bar{E})=\mu(E)$;
\item[(iv)] $\mu(conv{E})=\mu(E)$;
\item[(v)] $\mu\left[\lambda E_1+(1-\lambda)E_2\right]\leq\lambda \mu(E_1)+(1-\lambda)\mu(E_2)$  for $0\leq \lambda\leq 1$;
\item[(vi)] If $(E_n)$ is a sequence of closed sets from $\mathfrak{M}_{X}$ such that $E_{n+1}\subset E_n$ and $\displaystyle \lim_{n\rightarrow\infty} \mu(E_n)=0$ then the intersection set $\displaystyle E_{\infty}=\bigcap_{n=1}^{\infty}E_n$ is non-empty.
\end{enumerate}  
\end{definition}
\noindent
Further properties of Hausdorff measure of noncompactness $\chi $ can be found in \cite{Josaf, M1-b}.\\

The fixed point theorem of Darbo \cite{darbo55} that is used in present paper states that:
\begin{lemma} \label{lemma}\cite{darbo55} Let $E$ be a non-empty, bounded, closed, and convex subset of Banach space $X$ and let $T:E\rightarrow E$ be a continuous mapping. Assume that there exists a constant $k\in [0,1)$ such that $\mu\left( T(E)\right)\leq k\mu(E)$ for any non-empty subset $E$ of $X$. Then $T$ has a fixed point in the set $E$.
\end{lemma}   
The idea of equicontinuous sets is defined as:
\begin{definition}[Equicontinuous]
Let $(\Omega_{1},d)$ and $(\Omega_{2},d)$ be two metric spaces, and $\mathcal{T}$ the family of functions from $\Omega_{1}$ to $\Omega_{2}$. The family $\mathcal{T}$ is equicontinuous at a point $m_0\in \Omega_{1}$ if for every $\epsilon>0$, there exists a $\delta>0$ such that 
$d(f(m),f(m_{0}))<\epsilon$  for all $f\in \mathcal{T}$ and all $m\in \Omega_{1}$  such that $d(m,m_{0})<\delta$. The family is pointwise equicontinuous if it is equicontinuous at each point of $\Omega_{1}$.
\end{definition}
For fixed $p~,~p\geq 1$, we denote by $\ell_{p}$ the Banach sequence space with $\|\cdot \|_{p}$ norm defined as:
$$\|x\|_{p}=\|(x_{n})\|_{p}=\left( \sum_{n=1}^{\infty}|x_{n}|^{p}\right)^{1\over p} $$
for $x=(x_n)\in \ell_{p}.$ In order to apply Lemma \ref{lemma} in a given Banach space $X$, we need a formula expressing the measure of noncompactness by a simple formula. Such formulas are known only in a few spaces \cite{Josaf, M1-b}.\\  
For the Banach sequence space $\left( \ell_p, \|\cdot\|_{p}\right)$, Hausdorff measure of noncompactness is given by
\begin{equation}\label{Mes-lp}
\chi(E)=\lim_{n\rightarrow\infty }\left\{\sup_{ (e_k)\in E}\left( \sum_{k\geq n}|e_k|^{p}\right)^{1\over p} \right\}
\end{equation}
where $E\in \mathfrak{M}_{\ell_{p}}$. The above formulas will be used in the sequel of the paper.\\

In recent years many researchers have worked on the infinite system of second order differential equations of the form 
\begin{equation}\label{1}
u_{i}''=-f_{i}(t,u_1,u_2,\ldots )~~,~~~~~~~~~~~~~~~~~~~~u_{i}(0)=u_{i}(T)=0~,~i\in \mathbb{N}~,~ t\in [0,T]
\end{equation}
and obtained the conditions for the existence of solutions of \eqref{1} in different Banach spaces \cite{aghajani2015application, banas2017existence, mursaleen2016solvability, mursaleen2017solvability}.\\
 Measure of noncompactness has been used to obtain conditions under which an infinite system of differential equations has a solution in given Banach space \cite{aghajani2015application, banas2001solvability, M1-b, banas2017existence, mursaleen2012applications, mursaleen2017solvability,  mursaleen2013application}.\\

We consider the infinite system of second order differential equations of the form 
\begin{equation}\label{basicequation}
\frac{\mathrm{d}^2v_{j}}{\mathrm{d}t^2}+v_{j}=f_{j}(t,v(t))
\end{equation}
where $t\in [0,T]$, $v(t)=\left(v_{j}(t)\right)_{j=1}^{\infty}$ and $j=1,2,\ldots $ .\\
The above system will be studied together with the boundary problem
\begin{equation} \label{bvp}
v_{j}(0)=v_{j}(T)=0
\end{equation}
The solution is investigated using the infinite system of integral equations and Green's function \cite{greensfunction}. Such systems appear in the study of theory of neural sets, theory of branching process and theory of dissociation of polymers \cite{deimling2006ordinary, deimling2010nonlinear}.\\ 

In this paper, we find the conditions under which the system given in \eqref{basicequation} under the boundary condition \eqref{bvp} has solution in the Banach sequence space $\ell_{p}$ to do so we define an equivalent infinite system of integral equations. The result is supported by an example.

\section{Main Results}

By $C(I,\mathbb{R})$ we denote the space of continuously differentiable functions on $I=[a,b]$ and by $C^{2}(I,\mathbb{R})$ the space of twice continuously differentiable functions on $I=[a,b]$. A function $v\in C^{2}(I,\mathbb{R})$ is a solution of \eqref{basicequation} if and only if $v$ is a solution of the infinite system of integral equations 
\begin{equation}\label{Inteqn}
v_{j}(t)=\int_{0}^{T}G(t,s)f_{j}(s,v(s))ds~,~~~t\in I
\end{equation} 
where $f_{j}(t,v) \in C(I,\mathbb{R})~,~j=1,2,3,\ldots$ and the Green's function $G(s,t)$ defined on the square $I^2$ as: 
\begin{equation}\label{Green}
G(t,s)=\left\{\begin{matrix}
\frac{\sin (t)\sin(T-s)}{\sin (T)}&:&0\leq s< t\leq T,\\ \\
\frac{\sin (s)\sin(T-t)}{\sin (T)}&:&0\leq t<s\leq T. 
\end{matrix}\right.
\end{equation}
This function satisfies the inequality 
\begin{equation} \label{Green1}
G(t,s)\leq {1\over 2}\tan \left(0.5T\right)
\end{equation}
for all $(t,s)\in I^2.$\\ 
From \eqref{Inteqn} and \eqref{Green}, we obtain  
\begin{align*}
v_{j}(t)=\int_{0}^{t}\frac{\sin (t)\sin(T-s)}{\sin (T)}f_{j}(s,v(s))ds+\int_{t}^{T}\frac{\sin (s)\sin(T-t)}{\sin (T)}f_{j}(s,v(s))ds.
\end{align*}
Differentiation gives 
\begin{align*}
\frac{\mathrm{d}v_{j}}{\mathrm{d}t}=\int_{0}^{t}\frac{\cos (t)\sin(T-s)}{\sin (T)}f_{j}(s,v(s))ds+\int_{t}^{T}\frac{-\sin (s)\cos(T-t)}{\sin (T)}f_{j}(s,v(s))ds
\end{align*}
Again, differentiating gives 
\begin{align*}
\frac{\mathrm{d}^2v_{j}}{\mathrm{d}t^2}=&\int_{0}^{t}\frac{-\sin (t)\sin(T-s)}{\sin (T)}f_{j}(s,v(s))ds+\frac{\cos (t)\sin(T-t)}{\sin (T)}f_{j}(t,v(t))+\\
&\int_{t}^{T}\frac{-\sin (s)\sin(T-t)}{\sin (T)}f_{j}(s,v(s))ds+\frac{\sin (t)\cos(T-t)}{\sin (T)}f_{j}(t,v(t))\\
&=-\int_{0}^{T}G(t,s)f_{j}(s,v(s))ds+{1\over \sin(T)}\left[\sin (t)\cos(T-t)+\cos (t)\sin(T-t) \right]f_{j}(t,v(t))\\
&=-v_{j}(t)+f_{j}(t,v(t)).
\end{align*}
Thus $v_{j}(t)$ given in \eqref{Inteqn} satisfies \eqref{basicequation}. Hence finding existence of solution for the system \eqref{basicequation} with boundary conditions \eqref{bvp} is equivalent to find the existence of solution for the infinite system of integral equations \eqref{Inteqn}. 
\begin{remark}\label{remark}
If $X$ is a Banach space and $\chi_{X}$ denotes its Hausdorff measure of noncompactness, then Hausdorff measure of noncompactness of a subset $E$ of $C(I,X)$, the Banach space of continuous functions is given by \cite{Josaf, malkowsky1} 
$$\chi(E)=\sup \left\{\chi_{X}(X(t)):t\in I \right\}. $$
where $E$ is equicontinuous on the interval $I=[0,T].$
\end{remark}

In order to find the condition under which the system \eqref{Inteqn} has a solution in $\ell_{p}$ we need the following assumptions:
\begin{enumerate}
\item[$(\bold{A}_1)$] The functions $f_{j}$ are real valued, defined on the set $I\times \mathbb{R}^{\infty}$, $(j=1,2,3,\ldots)$.
\item[$(\bold{A}_2)$] An operator $f$  defined on the space $I\times \ell_{p}$ as 
$$(t,v)\mapsto \left(fv\right)(t)= \left(f_{j}(t,v)\right)=\left(f_{1}(t,v),f_{2}(t,v),f_{3}(t,v),\ldots\right)$$
transforms the space $I\times \ell_{p}$ into $\ell_{p}$.\\
The class of all functions $\left\{ \left(fv\right)(t)\right\}_{t\in I}$ is equicontinuous at each point of the space $\ell_{p}$.
That is for each $v\in \ell_{p}$, fixed arbitrarily and given $\epsilon>0$ there exists  $\delta>0$ such that whenever $\|u-v\|_{p}<\delta$ 
\begin{equation} \label{Eqcont}
\|(fu)(t)-(fv)(t)\|_p<\epsilon
\end{equation} 
\item[$(\bold{A}_3)$] For each $t\in I, ~v(t)=\left(v_{j}(t)\right)\in \ell_{p}$, the following inequality holds
\begin{equation}\label{Ineq1}
\left|f_{j}(t, v(t)) \right|^{p}\leq g_{j}(t)+h_{j}(t)|v_{j}|^{p} \hspace{1.5cm}n\in\mathbb{N} 
\end{equation} 
where $h_{j}(t)$ and $g_{j}(t)$ are real valued continuous functions on $I$. The function $g_{j}~(j=1,2,\ldots)$ is continuous on $I$ and the function series $\displaystyle \sum_{k\geq 1}g_{k}(t)$ is uniformly convergent. Also the function sequence $\left(h_{j}(t)\right)_{j\in \mathbb{N}}$ is equibounded on $I$. 
\end{enumerate}
To prove the general result we set the following constants
\begin{align*}
g(t)&=\sum_{j=1}^{\infty} g_{j}(t),\\
G&=\max \left\{ g(t): t\in I\right\},\\
H&=\sup\left\{ h_{j}(t):t\in I,j\in\mathbb{N}\right\}.
\end{align*}     
\begin{theorem} \label{Th1}
Under the assumptions $(\bold{A}_{1})-(\bold{A}_{3})$, with  $\displaystyle (HT)^{1\over p}\tan(0.5 T)<2$, $T\not=(2n-1)\pi, n=1,2,\ldots$ the infinite system of integral equations \eqref{Inteqn} has atleast one solution $v(t)=(v_{j}(t))$ in $\ell_{p}$ space $p\geq 1$, for each $t\in I.$
\end{theorem}
\begin{proof}
We consider the space $C(I,\ell_{p})$ of all continuous functions on $I=[0,T]$ with supremum norm given as:
$$\| v\|=\sup_{t\in I}\left\{ \|v(t)\|_{p}\right\}. $$
Define the operator $\mathcal{F}$ on the space $C(I,\ell_{p})$ by 
\begin{align} \label{Opr}
\begin{split}
\left(\mathcal{F}v\right)(t)&=\left( (\mathcal{F}v)_{j}(t)\right)\\
&=\left(\int_{0}^{T}G(t,s)f_{j}(s,v(s)) ds\right)\\ 
&=\left(\int_{0}^{T}G(t,s)f_{1}(s,v(s))ds,\int_{0}^{T}G(t,s)f_{2}(s,v(s))ds,\ldots \right).
\end{split}
\end{align} 
The operator $\mathcal{F} $ as defined in \eqref{Opr} transforms the space $C(I,\ell_{p})$ into itself, which we will show. Fix  $v=v(t)=\left(v_{j}(t)\right)$ in $C(I,\ell_{p})$ then for arbitrary $t\in I$ using assumption $(\bold{A}_3)$, inequality  \eqref{Green1} and  H$\ddot{o}$lder's inequality we have 
\begin{align*}
\left(\|(\mathcal{F}v)(t) \|_{p} \right)^{p}&=\sum_{j=1}^{\infty}\left|G(t,s)f_{j}(s,v(s))ds \right|^{p}\\
&\leq\sum_{j=1}^{\infty}\left\{\int_{0}^{T} |G(t,s)|^{p}|f_{j}(s,v(s))|^{p}ds\right\}\left(\int_{0}^{T}ds\right)^{p\over q}\\
&\leq \left(T\right)^{p\over q}\sum_{j=1}^{\infty}\left\{\int_{0}^{T} |G(t,s)|^{p} \left[g_{j}(s)+h_{j}(s)|v_{j}(s)|^{p} \right]ds\right\}\\
&\leq \left({1\over 2} \tan(0.5T)\right)^{p} \left(T\right)^{p\over q}\sum_{j=1}^{\infty} \left[\int_{0}^{T}g_{j}(s)ds+\int_{0}^{T}h_{j}(s)|v_{j}(s)|^{p}ds\right].
\end{align*}  
Now using Lebesgue dominated convergence theorem we get 
\begin{align*}
\left(\|(\mathcal{F}(v)(t) \|_{p} \right)^{p}&\leq \left({T^{1\over q}\over 2}\tan(0.5T)\right)^{p}\left( \int_{0}^{T}g(s) ds+H \int_{0}^{T}\sum_{j=1}^{\infty}|v_{j}(s)|^{p}ds\right)\\
&\leq \left({T^{1\over q}\over 2}\tan(0.5T)\right)^{p}\left(GT+HT\left (\|v\|_{p}\right)^{p}\right)\\
&=\left({T\over 2}\tan(0.5T)\right)^{p}\left(G+H\left (\|v\|_{p}\right)^{p}\right).
\end{align*}
Therefore 
\begin{equation} \label{normin}
\left(\|(\mathcal{F}(v)(t) \|_{p} \right)^{p}\leq \left({T\over 2}\tan(0.5T)\right)^{p}\left(G+H\left (\|v\|_{p}\right)^{p}\right).
\end{equation}
Hence $\mathcal{F}v$ is bounded on the interval $I$. Thus $\mathcal{F}$ transforms the space $C(I,\ell_{p})$ into itself. From \eqref{normin} we get 
\begin{equation}\label{Norm2}
\|(\mathcal{F}(v)(t) \|_{p}\leq {T\over 2}\tan(0.5T)\left(G+H\left (\|v\|_{p}\right)^{p}\right)^{1\over p}.
\end{equation}
Now using \eqref{Inteqn} and following the procedure as above we get 
\begin{align*}
\left(\|v\|_{p}\right)^{p}&\leq \left({T\over 2}\tan(0.5T)\right)^{p}\left(G+H\left (\|v\|_{p}\right)^{p}\right)\\
\Rightarrow \left(\|v\|_{p}\right)^{p}&\leq {G\left( T\tan(0.5T)\right)^{p} \over 2^{p}-H(T\tan (0.5T))^{p}}\\
\Rightarrow \|v\|_{p}&\leq {G^{1\over p}\left( T\tan(0.5T)\right) \over \left[ 2^{p}-H(T\tan (0.5T))^{p}\right]^{1\over p}}.
\end{align*}
Thus the positive number 
$$\displaystyle r={G^{1\over p}\left( T\tan(0.5T)\right) \over \left[ 2^{p}-H(T\tan (0.5T))^{p}\right]^{1\over p}}$$
is the optimal solution of the inequality 
$$ {T\over 2}\tan(0.5T)\left(G+HR^{p}\right)^{1\over p} \leq R.$$
Hence, by \eqref{Norm2} the operator $\mathcal{F}$ transforms the ball $B_{r}\subset C(I,\ell_{p})$ into itself.\\

We know show that $\mathcal{F}$ is continuous on $B_{r}$. Let $\epsilon>0$ be arbitrarily fixed and $v=(v(t))\in B_{r}$ be any arbitrarily fixed function, then if $u=(u(t))\in B_{r}$ is any function such that $\| u-v\|<\epsilon$, then for any $t\in I$, we have 
   
\begin{align*}
\left( \|(\mathcal{F}u)(t)-(\mathcal{F}v)(t)\|_{p}\right)^{p}&=\sum_{j=1}^{\infty} \left| \int_{0}^{T}G(t,s)\left[f_{j}(s,u(s))-f_{j}(s,v(s)) \right]ds \right|^{p}\\
&\leq \sum_{j=1}^{\infty}  \int_{0}^{T}\left| G(t,s)\right|^{p} \left|f_{j}(s,u(s))-f_{j}(s,v(s)) \right|^{p}ds\left(\int_{0}^{T} ds\right)^{p\over q}\\
&\leq \left(T\right)^{p\over q}\sum_{j=1}^{\infty} \int_{0}^{T} |G(t,s)|^{p} \left|f_{j}(s,u(s))-f_{j}(s,v(s)) \right|^{p}ds.
\end{align*}
Now, by using \eqref{Green1} and the assumption $(\bold{A}_2)$ of equicontinuity we get
\begin{small}
\begin{align} \label{A1}
\begin{split}
\left( \|(\mathcal{F}u)(t)-(\mathcal{F}v)(t)\|_{p}\right)^{p}&\leq \left(T\right)^{p\over q}\left({1\over 2}\tan (0.5T)\right)^{p} \sum_{j=1}^{\infty} \int_{0}^{T} \left|f_{j}(s,u(s))-f_{j}(s,v(s)) \right|^{p}ds\\ 
&= \left({T^{1\over q}\over 2}\tan(0.5T)\right)^{p}\lim_{m\rightarrow \infty } \sum_{j=1}^{m}\int_{0}^{T} \left|f_{j}(s,u(s))-f_{j}(s,v(s)) \right|^{p}ds\\
&= \left({T^{1\over q}\over 2}\tan(0.5T)\right)^{p}\lim_{m\rightarrow \infty }\int_{0}^{T} \left(\sum_{j=1}^{m}\left|f_{j}(s,u(s))-f_{j}(s,v(s)) \right|^{p}\right) ds.
\end{split}
\end{align}
\end{small}  
Define the function $\delta(\epsilon)$ as 
$$\delta(\epsilon)=\sup \left\{|f_{j}(s,u(s))-f_{j}(s,v(s))|:u,v\in\ell_{p},\|u-v\|\leq \epsilon t\in I, j\in \mathbb{N} \right\}.$$ 
Then clearly $\delta(\epsilon)\rightarrow 0$ as $\epsilon \rightarrow 0$ since the family $\displaystyle \left\{ (fv)(t):t\in I\right\}$ is equicontinuous at every point $v\in\ell_{p}$.\\
Therefore, by \eqref{A1} and using Lebesgue dominant convergence theorem, we have 
\begin{align*}
 \left( \|(\mathcal{F}u)(t)-(\mathcal{F}v)(t)\|_{p}\right)^{p}&\leq \left({T^{1\over q}\over 2}\tan(0.5T)\right)^{p}\int_{0}^{T} \left[\delta(\epsilon)\right]^{p}ds\\
 &=\left({T\over 2}\tan(0.5T)\right)^{p}\left[\delta(\epsilon)\right]^{p}
\end{align*} 
This implies that the operator $\mathcal{F}$ is continuous on the ball $B_{r}$, since $T\not=(2n-1)\pi~,n=1,2,\ldots $ .\\
 
Since $G(t,s)$ as defined in \eqref{Green} is uniformly uniformly continuous on $I^2$, so by definition of operator $\mathcal{F}$, it is easy to show that $\left\{\mathcal{F}u:u\in B_r\right\}$ is equicontinuous on $I$. Let $\displaystyle B_{r_{1}}=conv(\mathcal{F}B_{r})$, then $\displaystyle B_{r_{1}}\subset B_{r}$ and the functions from the set $\displaystyle B_{r_{1}}$ are equicontinuous on $I$.\\

Let $E\subset \displaystyle B_{r_{1}}$, then $E$ is equicontinuous on $I$. If $v\in E$ is a function then for arbitrarily fixed $t\in I$, we have by assumption $(\bold{A}_3)$ 
\begin{align*}
\sum_{j=k}^{\infty} \left|\left(\mathcal{F}v\right)_{j}(t) \right|^{p}&=\sum_{j=k}^{\infty} \left|\int_{0}^{T}G(t,s)f_{j}(s,v(s))ds \right|^{p}\\
&\leq \sum_{j=k}^{\infty} \left(\int_{0}^{T}|G(t,s)||f_{j}(s,v(s))|\right)^{p} 
\end{align*}  
Using H$\ddot{o}$lder's inequality and \eqref{Green1} we get 
\begin{align*}
\sum_{j=k}^{\infty} \left|\left(\mathcal{F}v\right)_{j}(t) \right|^{p}&\leq \sum_{j=k}^{\infty} \left(\int_{0}^{T}|G(t,s)|^{p}|f_{j}(s,v(s))|^{p}ds\right)\left(\int_{0}^{T}ds\right)^{p\over q}\\
&\leq T^{p\over q}\left( {1\over 2}\tan (0.5T)\right)^{p}  \sum_{j=k}^{\infty} \left(\int_{0}^{T}|f_{j}(s,v(s))|^{p}ds\right)
\end{align*}
Using Lebesgue dominant convergence theorem and the assumption $(\bold{A}_3)$ gives 
\begin{align*}
\sum_{j=k}^{\infty} \left|\left(\mathcal{F}v\right)_{j}(t) \right|^{p}&\leq \left( {T^{1\over q}\over 2}\tan (0.5T)\right)^{p}\int_{0}^{T}\left\{ \sum_{j=k}^{\infty}[g_{j}(s)+h_{j}(s)|v_{j}(s)|^{p}]\right\}ds\\
&=  \left( {T^{1\over q}\over 2}\tan (0.5T)\right)^{p}\left\{\int_{0}^{T}\left(\sum_{j=k}^{\infty}g_{j}(s)\right) ds +\int_{0}^{T}\left(\sum_{j=k}^{\infty} h_{j}(s)|v_{j}(s)|^{p}\right)ds\right\} \\
&\leq \left( {T^{1\over q}\over 2}\tan (0.5T)\right)^{p}\left\{\int_{0}^{T}\left(\sum_{j=k}^{\infty}g_{j}(s)\right)ds +H\int_{0}^{T}\sum_{j=k}^{\infty} |v_{j}(s)|^{p}ds\right\}
\end{align*}
Taking supremum over all $v\in E$  we obtain 
\begin{align*}
\sup_{v\in E}\sum_{j=k}^{\infty} \left|\left(\mathcal{F}v\right)_{j}(t) \right|^{p}\leq \left( {T^{1\over q}\over 2}\tan (0.5T)\right)^{p}\left\{\int_{0}^{T}\left(\sum_{j=k}^{\infty}g_{j}(s)\right)ds +H\sup_{v\in E}\int_{0}^{T}\sum_{j=k}^{\infty} |v_{j}(s)|^{p}ds\right\}.
\end{align*}
Using the definition of Hausdorff measure of noncompactness in $\ell_{p}$ space and noting that $E$ is the set of equicontinuous functions on $I$, then by using Remark \ref{remark} we get 
\begin{align*}
\left(\chi (\mathcal{F}E)\right)^{p}&\leq HT \left({1\over 2}\tan(0.5 T)\right)^{p} \left(\chi(E)\right)^{p}\\
\Rightarrow \chi (\mathcal{F}E)&\leq (HT)^{1\over p}\left({1\over 2}\tan(0.5 T)\right)\chi(E).
\end{align*} 
Therefore if  $\displaystyle (HT)^{1\over p}\left({1\over 2}\tan(0.5 T)\right)<1$ that is $\displaystyle (HT)^{1\over p}\tan(0.5 T)<2$, then by Lemma \ref{lemma}, the operator $\mathcal{F}$ on the set $B_{r_{1}}$ has a fixed point, which completes the proof of the theorem.     
\end{proof}
Now the system of integral equations \eqref{Inteqn} is equivalent to the boundary value problem \eqref{basicequation}, we conclude that the infinite system of second order differential equations  \eqref{basicequation} satisfying the boundary conditions \eqref{bvp}, has atleast one solution $ v(t)=\left(v_{1}(t),v_{2}(t), \ldots\right) \in \ell_{p}$ such that $v_{j}(t)\in C^{2}(I,\ell_{p})~,~(j=1,2,\ldots) $ for any $t\in I$, if the assumptions of Theorem \ref{Th1} are satisfied. \\

\noindent 
\textbf{Note:} The value of $T$ is chosen in such that the condition $\displaystyle (HT)^{1\over p}\tan(0.5 T)<2$ is satisfied.\\ 
The above result is illustrated by the following example: 
\begin{example}
Consider the infinite system of second order differential equations in $\ell_{2}$
\begin{equation}\label{exampl}
\frac{\mathrm{d}^2v_{n}}{\mathrm{d}t^2}+v_{n}={t3^{-nt}\over n}+\sum_{k=n}^{\infty}\frac{\cos t}{(1+2n)\sqrt{(k-1)!}}\cdot \frac{v_{k}(t)[1-(k-n)v_{k}(t)]}{(k-n+1)}.
\end{equation}
for $n=1,2,\ldots $ .
\end{example}
\textbf{Solution:} Compare \eqref{exampl} with \eqref{basicequation} we have 
 \begin{equation}\label{examp2}
 f_{n}(t,v)={t3^{-nt}\over n}+\sum_{k=n}^{\infty}\frac{\cos t}{(1+2n)\sqrt{(k-1)!}}\cdot \frac{v_{k}(t)[1-(k-n)v_{k}(t)]}{(k-n+1)}
 \end{equation}
Assumption $(\bold{A}_1)$ of the Theorem \ref{Th1} is clearly satisfied. We now show that assumption $(\bold{A}_2)$ of the Theorem \ref{Th1} is also satisfied  that is 
\begin{equation}
|f_{n}(t,v)|^{2}\leq g_{n}(t)+h_{n}(t)|v_{n}|^{2}
\end{equation}
Using Cauchy-Schwarz inequality and equation \eqref{exampl} we have 
\begin{align*}
|f_{n}(t,v)|^{2}&=\left|{t3^{-nt}\over n}+\sum_{k=n}^{\infty}\frac{\cos t}{(1+2n)\sqrt{(k-1)!}}\cdot\frac{v_{k}(t)[1-(k-n)v_{k}(t)]}{(k-n+1)} \right|^{2}
\\
&\leq 
2\left\{{t^{2}3^{-2nt}\over n^{2}}+\left\{\sum_{k=n}^{\infty}\frac{\cos t}{(1+2n)\sqrt{(k-1)!}}\cdot\frac{v_{k}(t)[1-(k-n)v_{k}(t)]}{(k-n+1)}\right] \right\}
 \\
&\leq 
2{t^{2}3^{-2nt}\over n^{2}}+2\left(\sum_{k=n}^{\infty}\frac{\cos^2 t}{(1+2n)^2(k-1)!}\right)\cdot\sum_{k=n}^{\infty}\left(\frac{v_{k}(t)[1-(k-n)v_{k}(t)]}{(k-n+1)} \right)^{2}
\end{align*}
Now using the fact that  $\displaystyle {1-\alpha \beta\over \beta} \leq {1\over (2\beta)^{2}}$ for any real $\alpha, \beta$, $\beta\not=0$ we have 
\begin{align*}
|f_{n}(t,v)|^{2}&\leq 2{t^{2}3^{-2nt}\over n^{2}}+2\frac{\cos^2 t}{(1+2n)^2}\times e\times \left(v_{n}^{2}+\sum_{k=n+1}^{\infty} \frac{v_{k}(t)[1-(k-n)v_{k}(t)]}{(k-n+1)}  \right)\\
&\leq
2{t^{2}3^{-2nt}\over n^{2}}+2\frac{e[\cos^2 t]}{(1+2n)^2}(v_{n}^{2})+2\frac{e[\cos^2 t]}{(1+2n)^2}\times \sum_{k=n+1}^{\infty}\left(1\over 2(k-n)\right)^{2}\\
&\leq  2{t^{2}3^{-2nt}\over n^{2}}+{1\over 2} \frac{e[\cos^2 t]}{(1+2n)^2}\times {\pi^{2}\over 6}+2\frac{e[\cos^2 t]}{(1+2n)^2}(v_{n}^{2})
\end{align*}
Hence by taking 
\begin{align*}
g_{n}(t)=2{t^{2}3^{-2nt}\over n^{2}}+{\pi^{2}\over 12} \frac{e[\cos^2 t]}{(1+2n)^2}~~,~~h_{n}(t)=2\frac{e[\cos^2 t]}{(1+2n)^2}
\end{align*}
it is clear that $g_{n}(t)$ and $h_{n}(t)$ are real valued continuous functions on $I$. Also 
\begin{align*}
|g_{n}(t)|&\leq 2{T^2\over n^2}+{\pi^{2}\over 12} \frac{e}{(1+2n)^2}\\
&\leq 
\left(2T^2+{\pi^2e\over 12} \right){1\over n^2}
\end{align*}
for all $t\in I$. Thus by Weierstrass test for uniform convergence of the function series we see that $\displaystyle \sum_{k\geq 1} g_{k}(t)$ is uniformly convergent on $I$.\\
Further, we have
\begin{align*}
|h_{j}(t)|\leq {2e\over (1+2n)^2}
\end{align*}  
for all $t\in I$.\\
Thus the function sequence $\left(h_{j}(t)\right)$ is equibounded on $I$. Thus \eqref{examp2} is satisfied and hence the assumption $(\bold{A}_3)$ is satisfied.\\
\newpage
\noindent
Also
$$\displaystyle G=\sup\left\{ \sum_{k\geq 1}g_{k}(t):t\in I\right\}=\left(2T^2+{\pi^2e\over 12} \right){\pi^2\over 6}$$
and
 $$\displaystyle H=\sup\left\{ h_{j}(t):t\in I\right\}={2e\over 9}$$

The assumption $(\bold{A}_2)$ is also satisfied as for fixed $t\in T$ and $(v_{j}(t))=\left(v_{1}(t), v_{2}(t), \ldots \right)\in \ell_{2} $
we have
\begin{align*}
\sum_{j=1}^{\infty}\left|f_{j}(t,v) \right|^2&=\sum_{j=1}^{\infty}g_{j}(t)+\sum_{j=1}^{\infty}h_{j}(t) |v_{j}(t)|^{2}\\
&\leq G+H\sum_{j=1}^{\infty}|v_{j}(t)|^{2}
\end{align*}
Hence the operator $f=(f_{j})$ transforms the space $(I, \ell_{2})$ into $\ell_{2}$.\\
Also for $\epsilon>0$ and $u=(u_j), v=(v_j)$ in $\ell_{2}$ with $\| u-v\|_2<\epsilon$, we have 
\begin{align*}
\bigg(\|\left(fu\right)&(t)-\left(fv\right)(t)\|_{2}\bigg)^{2}=\sum_{n=1}^{\infty}\left|f_{n}(t,u(t)) -f_{n}(t,v(t))  \right|^2\\
&=
\sum_{n=1}^{\infty}\left\{\left|\sum_{k=n}^{\infty} \frac{(\cos t)u_{k}(t)[1-(k-n)u_{k}(t)]}{(1+2n)(k-n+1)\sqrt{(k-1)!}}-
\frac{(\cos t)v_{k}(t)[1-(k-n)v_{k}(t)]}{(1+2n)(k-n+1)\sqrt{(k-1)!}}\right|^2\right\}\\
&\leq 
\sum_{n=1}^{\infty}\left\{\left(1\over (1+2n)^2 \right)\left|\sum_{k=n}^{\infty} \frac{u_{k}(t)[1-(k-n)u_{k}(t)]-v_{k}(t)[1-(k-n)v_{k}(t)]}{(k-n+1)\sqrt{(k-1)!}}\right|^2\right\}\\
&\leq 
\sum_{n=1}^{\infty}\left\{\left(1\over (1+2n)^2 \right)\left[\sum_{k=n}^{\infty} \left|\frac{(u_{k}(t)-v_{k}(t))[1-(k-n)(u_{k}(t)+v_{k}(t))]}{\sqrt{(k-1)!}(k-n+1)}\right|\right]^2\right\}
\end{align*} 
Using Holder's inequality we get 
\begin{align*}
\bigg(\|&\left(fu\right)(t)-\left(fv\right)(t)\|_{2}\bigg)^{2}\\
&\leq \sum_{n=1}^{\infty}\left\{{1\over (1+2n)^2}\left(\sum_{k=n}^{\infty}{1\over (k-1)!} \right)\left[\sum_{k=n}^{\infty} \left|\frac{(u_{k}(t)-v_{k}(t))[1-(k-n)(u_{k}(t)+v_{k}(t))]}{(k-n+1)}\right|^2\right]\right\}\\
&\leq 
e\sum_{n=1}^{\infty}\left\{{1\over (1+2n)^2}\left[\sum_{k=n}^{\infty}|u_{k}(t)-v_{k}(t)|^{2} \left|\frac{1-(k-n)(u_{k}(t)+v_{k}(t))}{(k-n+1)}\right|^2\right]\right\}\\
&\leq 
e\sum_{n=1}^{\infty}\left\{{1\over (1+2n)^2}\left[\sum_{k=n}^{\infty}|u_{k}(t)-v_{k}(t)|^{2}\right]\right\}\\
&<
e \epsilon^{2}\sum_{n=1}^{\infty}\left\{{1\over (1+2n)^2}\right\} \\
&\leq e\left({\pi^2\over 8}\right)\epsilon^2 
\end{align*} 
Thus for any $t\in I$, we have 
 \begin{align*}
 \|\left(fu\right)(t)&-\left(fv\right)(t)\|_{2}<{\pi \epsilon \sqrt e\over 2\sqrt2}.
 \end{align*}
Therefore the family $\displaystyle \left\{(fv)(t): t\in I\right\}$ is equicontinuous.\\
Finally we see that the condition $\displaystyle (HT)^{1\over p}\tan(0.5 T)<2$ is satisfied for all $T\leq 2$.\\
So, by Theorem \ref{Th1} there exists at least one solution to given infinite system of differential equations \eqref{exampl} in $C(I, \ell_{2})$.

\bibliographystyle{mmn}
\bibliography{References}

\end{document}